\documentclass[11pt, psamsfonts]{amsart}

\usepackage[a4paper,margin=4cm]{geometry}
\usepackage{amsfonts,amssymb,amscd}
\usepackage{graphicx,mathrsfs} 
\usepackage{subfigure}
\newtheorem{theorem}{Theorem}[section]

\newtheorem{proposition}[theorem]{Proposition}

\newtheorem{definition}[theorem]{Definition}

\theoremstyle{remark}

\newtheorem{remark}[theorem]{\bf Remark}
\newtheorem{example}[theorem]{\bf Example}
\newtheorem{remarks}[theorem]{\bf Remarks}

\renewcommand{\leq}{\leqslant}
\renewcommand{\geq}{\geqslant}

\newcommand{\ptl}{\partial}

\newcommand{\wtilde}{\widetilde}

\newcommand{\rr}{\mathbb{R}}

\newcommand{\nn}{\mathbb{N}}
\newcommand{\C}{\mathscr{C}_2}


\numberwithin{equation}{section}

\begin{document}

\title[Minimizing bisections for the maximum relative diameter]
{Bisections of centrally symmetric planar convex bodies minimizing the maximum relative diameter}

\author[A. Ca\~nete]{Antonio Ca\~nete}
\address{Departamento de Matem\'atica Aplicada I \\ Universidad de Sevilla}
\email{antonioc@us.es}

\author[S. Segura Gomis]{Salvador Segura Gomis}
\address{Departamento de An\'alisis Matem\'atico \\ Universidad de Alicante}
\email{salvador.segura@ua.es}


\begin{abstract}
In this paper we study the bisections of a centrally symmetric planar convex body
which minimize the maximum relative diameter functional.
We give necessary and sufficient conditions for being a minimizing bisection,
as well as analyzing the behavior of the so-called standard bisection.
\end{abstract}

\maketitle

\section{Introduction}

Historically, the classical geometric functionals (perimeter, area, volume, inradius, circumradius, diameter and width), and
the relations between them, have been intensely studied, yielding a great variety of optimization problems~\cite{sa,cfg}.
Possibly, the most relevant example is the \emph{isoperimetric problem},
examining the relation between the area and the volume of sets in $\rr^n$ \cite{S,O}.
Moreover, in the setting of Convex Geometry, these functionals play an important role
and can be considered as the origin of this theory.

In this context, we shall focus on a particular relative geometric problem concerning the \emph{diameter} functional.
This is one of the most natural magnitudes for measuring the size of a set, and has been deeply considered in the literature.
Some well-known important results in $\rr^n$ regarding this functional are, for instance,
\emph{Jung's theorem} \cite{jung}, establishing the inequality between the diameter and the circumradius of a compact set,
the \emph{isodiametric inequality} \cite{bieberbach},
which asserts that the ball is the compact convex set of fixed volume with the minimum possible diameter,
or \emph{Borsuk's conjecture} \cite{borsuk}, asking whether any compact set $K$ can be divided into $n+1$ subsets
whose diameters are striclty less than the diameter of $K$.
Additionally, more inequalities involving the diameter and other classical functionals
for planar compact convex sets can be found in \cite{sa}.

In this work we shall consider the \emph{maximum relative diameter} functional in $\rr^2$,
which is defined
in the following way:
for a fixed planar compact convex set $C$,
a division of $C$ into two connected subsets determined by a simple
curve with endpoints in the boundary of $C$
will be called a \emph{bisection} of $C$.
Then, given a bisection $P$ of $C$ with subsets $C_1$, $C_2$ (not necessarily enclosing equal areas),
the maximum relative diameter associated to $P$ is
$$d_M(P)=\max\{D(C_1), D(C_2)\},$$
where $D(S)$ denotes the Euclidean diameter of a planar set $S$.
In view of this definition, the maximum relative diameter clearly represents the largest distance in the subsets generated by the bisection.
In this setting, we are interested in finding the \emph{optimal bisection} for the maximum relative diameter functional.
That is, among all the bisections of $C$,
we look for the one attaining the \emph{minimum possible value} for this functional, 
which can be considered as the division of $C$ with both subsets \emph{as small as possible} in terms of the diameter.

Throughout this paper, we shall assume that our sets are centrally symmetric.
This hypothesis provides enough geometric structure to deal with this problem,
allowing to obtain descriptive results for the minimizing bisections
(in the non-symmetric case, it seems not possible to find similar properties for the optimal bisections,
due to the lack of symmetry).
Moreover, considering divisions of the sets into \emph{two} subsets is something naturally inherent to central symmetry.
On the other hand, we point out that if we focus on bisections by straight lines in the class of compact convex sets,
the minimum value for the maximum relative diameter is precisely attained by a centrally symmetric set \cite[Th.~7]{mps}
so, in some sense, this kind of sets is certainly suitable for this functional.

Our main results refer to the minimizing bisections for the maximum relative diameter
of centrally symmetric planar compact convex sets.
We shall see that we do not have uniqueness of solution for this problem,
since proper slight modifications of a minimizing bisection will be also minimizing
(this is a common feature when working with the diameter functional).
Moreover, in order to find a minimizing bisection, our Proposition~\ref{prop:areas}
assures that it is enough to focus on the bisections given by a straight line passing
through the center of symmetry of the set
(which will always generate symmetric subsets enclosing equal areas).
This property agrees with the intuitive idea that the corresponding subsets of an optimal bisection
must be as \emph{balanced} as possible.
Proposition~\ref{prop:necesaria} shows a necessary condition for a bisection of this type to be minimizing
(expressed in terms of the farthest distances from its two endpoints),
which is complemented in our Theorem~\ref{th:main}, establishing a criterion for asserting that a bisection
(by a straight line passing through the center of symmetry) is minimizing.

This partitioning optimization problem has been already considered in \cite{mps},
but with the additional restriction of bisections with \emph{equal-area subsets}.
In this setting, among other results, it is proved that a minimizing bisection
is always given by a straight line passing through the center of symmetry of the set \cite[Prop.~4]{mps},
with no further description of the properties of the solutions.
Our work is inspired in this paper, with the aim of extending the results therein
to a more general situation (arbitrary bisections with non-equal area subsets),
and describing the minimizing bisections in a more precise way.

It is worth mentioning that the analogous question
for divisions into a \emph{larger number of subsets} has been also studied:
for a given a $k$-rotationally symmetric planar compact convex set $C$,
with $k\in\nn$, $k\geq3$,
a $k$-partition of $C$ is a decomposition of $C$  into $k$ connected subsets $C_1,\dots,C_k$,
determined by $k$ simple curves starting in $\ptl C$, all of them meeting in an interior point of $C$.
And similarly, given a $k$-partition $P$ of $C$, we can define the maximum relative diameter of $P$ as
$$d_M(P)=\max\{D(C_i):i=1,\dots,k\}.$$
In this context, we can investigate the $k$-partitions of $C$ attaining the minimum value for $d_M$.
This was treated in \cite{extending} (see also~\cite{trisecciones}), where it is
proved that the so-called \emph{standard $k$-partition}
(constructed by using $k$ inradius segments symmetrically placed, see Figure~\ref{fig:some})
is a solution for this problem, for any $k\geq 3$.

\begin{figure}[h]
  \includegraphics[width=0.73\textwidth]{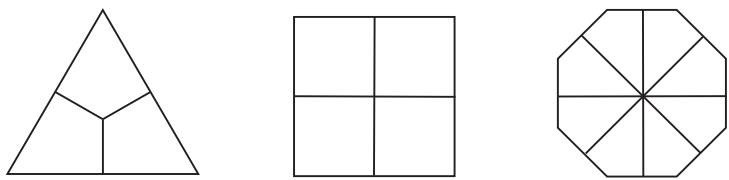}\\
  \caption{Some standard $k$-partitions for $k\geq 3$}\label{fig:some}
\end{figure}

\noindent We shall see in Section~\ref{se:standard} that the previous result does not hold in our setting
(which would correspond to $k=2$):
a \emph{standard bisection} (consisting of two symmetric inradius segments) is not minimizing in general
(for instance, see Example~\ref{ex:non}).
In fact, the standard bisection of a given centrally symmetric planar compact convex set
could not even be uniquely defined, as shown in Example~\ref{ex:capbody}.
These are two remarkable differences with respect to the case of
$k$-rotationally symmetric planar compact convex sets, with $k\geq 3$.

We have organized this paper as follows.
In Section~\ref{se:pre} we give the precise definitions and statement of our problem.
Section~\ref{se:main} contains our main results.
We prove in Proposition~\ref{prop:areas} that for any minimizing bisection,
we can find another bisection given by a straight line passing through the center of
symmetry of the set with the same value for the maximum relative diameter,
and therefore we can focus on this type of bisections in the search of a minimizing one.
Taking into account this,
Proposition~\ref{prop:necesaria} gives a necessary condition for minimizing bisections,
suggesting, in some sense, that it is needed the existence of a certain symmetry,
related to the farthest distances with respect to the endpoints.
Finally, Theorem~\ref{th:main} states some conditions for assuring
that a given bisection (by a straight line passing through the center of symmetry) is minimizing.
Unfortunately, this result is not completely sharp,
since there are examples of minimizing bisections whom Theorem~\ref{th:main} is not applicable, see Example~\ref{ex:6-gon}.
In Section~\ref{se:standard} we discuss the main features of the standard bisection,
showing that it is not minimizing in general (see Examples~\ref{ex:non} and~\ref{ex:salvador}),
and that it may not be uniquely defined (see Subsection~\ref{sub:uniqueness}).
Section~\ref{se:examples} contains several examples, showing some minimizing
bisections in each case by using Theorem~\ref{th:main}.
And we complete these notes with some comments of interest in Section~\ref{se:comments}.

\section{Preliminaries}
\label{se:pre}

Let us denote by $\mathscr{C}_2$ the class of centrally symmetric planar convex bodies
(recall that a body is, as usual, a compact set). The central symmetry of a set $C\in\C$
means that there exists a point $p\in C$ (called the \emph{center of symmetry} of $C$) such that
$C$ is invariant under the action of the rotation of angle $\pi$ centered at $p$.
Some examples of sets of this class are depicted in Figure~\ref{fig:examples}.

\begin{figure}[ht]
    \includegraphics[width=0.88\textwidth]{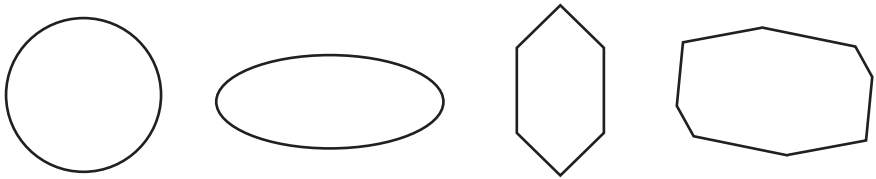}\\
  \caption{Some centrally symmetric planar convex bodies}
  \label{fig:examples}
\end{figure}

Throughout this paper, we shall focus on some particular divisions of our sets, called \emph{bisections}.
In Remark~\ref{re:justify} we shall justify that these are the most convenient divisions for our problem.
Note that the following definition can be done in a more general setting.

\begin{definition}
\label{de:bisection}
Let $C\in\C$.
A bisection of $C$ is a decomposition of $C$ into two connected
subsets, given by a simple curve with endpoints in the boundary $\ptl C$ of $C$.
\end{definition}

\begin{remark}
We point out that the curve determining a given bisection does not contain, in general, the center of symmetry of the set,
and moreover, the corresponding subsets do not enclose necessarily equal areas, as shown in Figure~\ref{fig:bisections}.
\end{remark}

\begin{figure}[h]
    \includegraphics[width=0.95\textwidth]{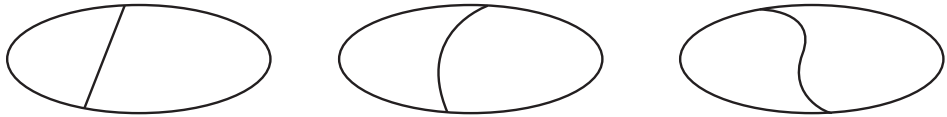}\\
  \caption{Three different bisections for an ellipse}\label{fig:bisections}
\end{figure}

We now proceed to define the geometric functional considered in this work,
previously introduced in \cite{mps}.

\begin{definition}
Let $C\in\C$, and let $P$ be a bisection of $C$, with associated subsets $C_1$, $C_2$.
The maximum relative diameter of $P$ is defined as
$$d_M(P)=\max\{D(C_1),D(C_2)\}, $$
where $D(S)$ denotes the Euclidean diameter of $S$.
\end{definition}

\begin{remark}
\label{re:boundary}
Recall that the diameter of a planar compact set is always attained by a pair of points lying in the boundary of the set,
and in the case of a polygon, by two of its vertices.
\end{remark}

For a fixed centrally symmetric planar convex body $C$, the purpose of these notes
is investigating the bisections of $C$ that \emph{minimize} the maximum relative diameter functional,
in the same spirit as in \cite{trisecciones,extending}, see also \cite{mps}:
determining these bisections precisely or, at least, describing some of their geometrical properties.
These bisections will be called \emph{minimizing} along this paper.
In this direction, some partial results have been obtained
in the case of bisections providing \emph{equal-area subsets}:
in this more restrictive setting, it has been proved that, for any set $C\in\C$, there always exists a minimizing bisection
given by a \emph{straight line passing through the center of symmetry} of the set~\cite[Prop.~4]{mps}.
However, no additional details have been outlined for the solutions, and nothing else is known.
We shall consider this problem in the most general setting
(that is, for bisections generating subsets which do not enclose necessarily the same quantity of area),
progressing in the description of these optimal bisections.

\begin{remark}
\label{re:justify}
We point out that a given set $C\in\C$ can be decomposed into two connected subsets
by means of divisions which are not bisections.
This can be done by using a simple closed curve \emph{entirely contained} in the interior of $C$.
In general, these decompositions are not good candidates for our problem
since, in view of Remark~\ref{re:boundary}, all of them have maximum relative diameter equal to $D(C)$,
which is an immediate upper bound for our functional.
Therefore, they will not be taken into account in these notes,
and we shall focus on the notion of bisection from Definition~\ref{de:bisection}.
\end{remark}

\begin{remark}
\label{re:uniqueness}
The \emph{uniqueness} of solution is not expected for this optimization problem,
as it usually occurs for questions involving the diameter functional.
In fact, if we have a minimizing bisection $P$ of a centrally symmetric planar convex body $C$,
slight modifications of $P$ can be done preserving the value of the maximum relative diameter, being minimizing as well.
This property suggests that a complete description of all the minimizing bisections of $C$ is not a feasible task.
\end{remark}

\subsection{Bisections by a straight line passing through the center of symmetry. }
\label{sub:lines}
Let $C\in\C$, and let $p$ be the center of symmetry of $C$.
The bisections of $C$ given by a straight line passing through $p$
possess some special properties and will play an important role for our problem.
Notice that for a bisection $P$ of this type,
the corresponding subsets $C_1$, $C_2$ will be \emph{congruent} due to the existing symmetry
(they will coincide up to the rotation of angle $\pi$ about $p$),
and so both of them will enclose the same quantity of area, and $D(C_1)=D(C_2)$.
Denoting by $v_1$, $v_2\in\ptl C$ the endpoints of the line segment determining $P$,
\cite[Prop.~3]{mps}
leads to
\begin{equation}
\label{eq:segment}
d_M(P)=\max\{d(v_1,x):x\in\ptl C\},
\end{equation}
where $d$ stands for the Euclidean distance in the plane.
Equality \eqref{eq:segment} implies that the maximum relative diameter of $P$
will be given by the distance between an endpoint of $P$ and any of its corresponding farthest points in $\ptl C$.

Apart from this, for a bisection $P$ determined by a straight line passing through $p$,
there is another equivalent expression for computing $d_M(P)$,
which will be useful along this work.
For any $x\in\ptl C_1$, we shall denote by $F_{C_1}(x)$ the set of farthest points from $x$ in $\ptl C_1$
(notice that $F_{C_1}(x)$ is non-empty due to compactness, and it may reduce to a single point).
Since $D(C_1)=D(C_2)$, we can just focus on \emph{one} of the subsets provided by $P$,
and using again~\cite[Prop.~3]{mps}, we will have that
\begin{equation}
\label{eq:farthest}
d_M(P)=D(C_1)=\max\{d(v_1,\phi_{C_1}(v_1)),\, d(v_2,\phi_{C_1}(v_2))\},
\end{equation}
where $v_1$, $v_2$ are the endpoints of $P$, and $\phi_{C_1}(v_i)\in F_{C_1}(v_i)$, $i=1,2$.

\begin{remark}
We note that it is easy to check that equalities~\eqref{eq:segment} and~\eqref{eq:farthest}
are not true neither for bisections given by a general planar curve,
nor by a straight line which does not pass through the center of symmetry, see Figure~\ref{fig:neq}.
\end{remark}

\begin{figure}[h]
  \includegraphics[width=0.97\textwidth]{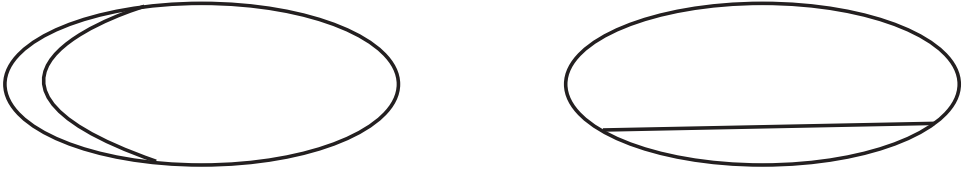}\\
    \caption{Equalities~\eqref{eq:segment} and~\eqref{eq:farthest} do not hold for these two bisections of the ellipse} \label{fig:neq}
\end{figure}

\begin{remark}
For a given $C\in\C$, and an arbitrary bisection $P$ of $C$
(not necessarily determined by a straight line),
with endpoints $v_1$, $v_2\in\ptl C$,
it may happen that $v_1\in F_{C_1}(v_2)$ and $v_2\in F_{C_1}(v_1)$.
In that case, it turns that $d_M(P)=d(v_1,v_2)$, in view of \eqref{eq:farthest},
and moreover, $C$ will be contained in the symmetric lens $B(v_1,d_M(P))\cap B(v_2,d_M(P))$,
where $B(x,r)$ denotes the Euclidean ball with center $x$ and radius $r$.
\end{remark}

\section{Main results}
\label{se:main}

In this section we obtain the main results of this paper.
Proposition~\ref{prop:areas}, which is an extension of~\cite[Prop.~4]{mps},
shows that there is always a minimizing bisection given by a straight line
passing through the center of symmetry of the set.
Proposition~\ref{prop:necesaria} states a necessary condition for a bisection to be minimizing,
and Theorem~\ref{th:main} establishes some conditions which allow to assert that a given bisection is minimizing.
These last two results (which are proved for bisections given by a straight line
passing through the center of symmetry)
reveal some of the geometric restrictions for being optimal.

\begin{proposition}
\label{prop:areas}
Let $C\in\C$, and let $p$ be the center of symmetry of $C$.
Let $P$ be a minimizing bisection for $d_M$ (whose subsets do not enclose necessarily equal areas).
Then there exists a bisection $P'$ given by a straight line passing through $p$ such that $d_M(P)=d_M(P')$.
\end{proposition}

\begin{proof}
Let $C_1$, $C_2$ be the subsets determined by $P$,
and let $v_1$, $v_2$ be the endpoints of $P$ (notice that $v_1\in \overline{C_1}\cap\overline{C_2}$).
We can assume that $d_M(P)=D(C_1)\geq D(C_2)$.
Let $v_1'\in\ptl C$ be the symmetric point of $v_1$ with respect to $p$,
and consider the bisection $P'$ given by the segment $\overline{v_1\,v_1'}$ (which passes through $p$).

Taking into account~\eqref{eq:segment}, we have that $d_M(P')=d(v_1,z)$, for certain $z\in\ptl C$.
If $z\in\ptl C_1$, then $d(v_1,z)\leq D(C_1)=d_M(P)$. And if $z\in\ptl C_2$, then $d(v_1,z)\leq D(C_2)\leq D(C_1)=d_M(P)$.
Thus $d_M(P')=d(v_1,z)\leq d_M(P)$, which implies that $d_M(P')=d_M(P)$ since $P$ is minimizing.
\end{proof}

\begin{remark}
\label{re:line}
A consequence of Proposition~\ref{prop:areas} is that,
in order to find a minimizing bisection for a centrally symmetric planar convex body,
we can focus on bisections given by a straight line passing through the corresponding center of symmetry.
Note that for these bisections, the \emph{endpoints} are always \emph{symmetric}
with respect to the center of symmetry of the set.
\end{remark}

\begin{remark}
\label{re:refinement}
In fact, Proposition~\ref{prop:areas} shows that, if $C\in\C$ and $p$ denotes its center of symmetry,
for any bisection $P$ of $C$ we can find another bisection $P'$,
given by a straight line passing through $p$, with $d_M(P')\leq d_M(P)$.
\end{remark}

Next Proposition~\ref{prop:necesaria} states a necessary condition for a bisection
(given by a straight line passing through the center of symmetry, in view of Remark~\ref{re:line}) to be minimizing,
by means of the farthest points of the endpoints of the bisection.
In some sense, this result suggests a certain balance for the optimal divisions:
the distances between \emph{each} endpoint and its corresponding farthest point
must coincide (being also equal to the value of the maximum relative diameter, due to equality~\eqref{eq:farthest}).

\begin{proposition} (Necessary condition)
\label{prop:necesaria}
Let $C\in\C$, and let $p$ be the center of symmetry of $C$.
Let $P$ be a bisection of $C$ given by a straight line passing through $p$,
with endpoints $v_1$, $v_2\in\ptl C$, and subsets $C_1$, $C_2$. 
If $P$ is a minimizing bisection for $d_M$, then
\begin{equation}
\label{eq:necessary}
d_M(P)=d(v_1,\phi_{C_1}(v_1))=d(v_2,\phi_{C_1}(v_2)),
\end{equation}
where $\phi_{C_1}(v_i)\in F_{C_1}(v_i)$.
\end{proposition}

\begin{proof}
Assume that $d_M(P)=D(C_1)=d(v_1,\phi_{C_1}(v_1))>d(v_2,\phi_{C_1}(v_2))$.
By applying a slight rotation centered at $p$ to the straight line containing the segment $\overline{v_1\,v_2}$,
it is clear that we can consider a new bisection $\wtilde{P}$,
with new endpoints $\wtilde{v_1}$, $\wtilde{v_2}\in\ptl C$ and subsets $\wtilde{C_1}$, $\wtilde{C_2}$,
satisfying that $\wtilde{v_1}\in\ptl C_1$.
By construction, we shall have that $d(v_1,\phi_{C_1}(v_1))>d(\wtilde{v_1},\phi_{C_1}(v_1))$.
Due to the continuity of the Euclidean distance, as $\wtilde{v_i}$ is close to $v_i$, $i=1,2$,
it follows that the inequality
$d(\wtilde{v_1},\phi_{\wtilde{C_1}}(\wtilde{v_1}))>d(\wtilde{v_2},\phi_{\wtilde{C_1}}(\wtilde{v_2}))$
will be preserved, at least, infinitesimally.
This implies that $d_M(\wtilde{P})=d(\wtilde{v_1},\phi_{\wtilde{C_1}}(\wtilde{v_1}))$, by using~\eqref{eq:farthest}.
Moreover, $d(\wtilde{v_1},\phi_{\wtilde{C_1}}(\wtilde{v_1}))$ will be close to $d(\wtilde{v_1},\phi_{C_1}(v_1))$.
Thus,
$$d_M(\wtilde{P})=d(\wtilde{v_1},\phi_{\wtilde{C_1}}(\wtilde{v_1}))<d(v_1,\phi_{C_1}(v_1))=d_M(P),$$
which contradicts the minimizing character of $P$.
\end{proof}

\begin{remarks}
We shall mention some brief comments concerning Proposition~\ref{prop:necesaria}.
\begin{itemize}
\addtolength{\itemsep}{2mm}

\item[i)] The reverse of Proposition~\ref{prop:necesaria} does not hold in general:
this can be seen by considering, for instance, a rectangle and the bisection given by the orthogonal line to the shortest edges
passing through the center of symmetry, which satisfies \eqref{eq:necessary} but it is clearly not minimizing.
Therefore, some additional hypotheses are needed for an eventual sufficient condition.

\item[ii)] A geometric interpretation of this result
is that the maximum relative diameter of a minimizing bisection
(given by a straight line passing through the center of symmetry)
is necessarily provided by at least \emph{two different segments} in each congruent subset,
unless it is uniquely achieved by the distance between the endpoints of the bisection.

\item[iii)] The reader may compare Proposition~\ref{prop:necesaria} with \cite[Prop.~3]{mps}:
in the case of a \emph{minimizing} bisection given by a straight line
passing through the center of symmetry,
the farthest distances from \emph{both} endpoints must coincide,
providing the value of the maximum relative diameter.
\end{itemize}
\end{remarks}

We shall now prove our main Theorem~\ref{th:main},
which establishes some conditions to assert that a given bisection is minimizing.

\begin{theorem}
\label{th:main}
Let $C\in\C$, and let $p$ be the center of symmetry of $C$. 
Let $P$ be a bisection of $C$ given by a straight line passing through $p$, with endpoints $v_1$, $v_2$, and subsets $C_1$, $C_2$.
If there exist $\phi_{C_1}(v_1)\in F_{C_1}(v_1)$, $\phi_{C_1}(v_2)\in F_{C_1}(v_2)$ such that
\begin{itemize}
\addtolength{\itemsep}{2mm}
\item[i)] $d(v_1,\phi_{C_1}(v_1))=d(v_2,\phi_{C_1}(v_2))$, and

\item[ii)] $\ptl C_2\subset A_1\cup A_2$, where $A_i$ is the complement in the plane of the Euclidean ball
$B_i=B(\phi_{C_1}(v_i),d(v_i,\phi_{C_1}(v_i)))$, for $i=1,2$,
\end{itemize}
then $P$ is a minimizing bisection for $d_M$.
\end{theorem}

\begin{proof}
Notice that $d_M(P)=d(v_1,\phi_{C_1}(v_1))=d(v_2,\phi_{C_1}(v_2))$,
in view of \eqref{eq:farthest} and the assumed hypothesis.
Consider now any bisection $\wtilde{P}$ of $C$
determined by a straight line passing through $p$,
with endpoints $\wtilde{v_1}$, $\wtilde{v_2}$.
One of these endpoints, say $\wtilde{v_2}$, will necessarily lie in $\ptl C_2$,
and so $\wtilde{v_2}\in A_1\cup A_2$.
Without loss of generality, we can assume that $\wtilde{v_2}\in A_1$.
Then, $d(\wtilde{v_2},\phi_{C_1}(v_1))\geq d(v_1,\phi_{C_1}(v_1))$, and so
$$d_M(\wtilde{P})\geq d(\wtilde{v_2},\phi_{C_1}(v_1))\geq d(v_1,\phi_{C_1}(v_1))=d_M(P),$$
which yields the minimizing character of $P$, taking into account Proposition~\ref{prop:areas}.
\end{proof}

\begin{remarks}
\label{re:re}
Regarding Theorem~\ref{th:main}, we point out the following comments: 
\begin{itemize}
\item[i)] The second hypothesis is equivalent to $\ptl C_2\cap(B_1\cap B_2)=\emptyset$, with the notation therein.
\item[ii)] It is not difficult to check that the second hypothesis implies that $d_M(P)\neq d(v_1,v_2)$.
\item[iii)] It may happen that $F_{C_1}(v_1)\cap F_{C_1}(v_2)$ is a non-empty set.
In that case, Theorem~\ref{th:main} can be applied trivially
and the corresponding bisection is minimizing.
\end{itemize}
\end{remarks}

\begin{remark}
We stress that, in order to apply Theorem~\ref{th:main},
we need to find \emph{appropriate} farthest points from the endpoints of the bisection.
That is, the hypotheses of Theorem~\ref{th:main} may not hold for \emph{all} possible choices for the corresponding farthest points,
as shown in the following example.
Consider a rhombus $C$ formed by joining two congruent equilateral triangles,
and the bisection given by the common edges, with endpoints $v_1$, $v_2\in\ptl C$, see Figure~\ref{fig:rombo}.
It is clear that $v_1\in F_{C_1}(v_2)$ and $v_2\in F_{C_1}(v_1)$,
but Theorem~\ref{th:main} cannot be used with those elections.
However, the vertex $q$ of $C$ belongs to $F_{C_1}(v_1)\cap F_{C_1}(v_2)$,
and so it is possible to apply the result for that farthest point, see Remarks~\ref{re:re},\ iii).
\end{remark}
\vspace{-5mm}

\begin{figure}[h]
   \includegraphics[width=0.5\textwidth]{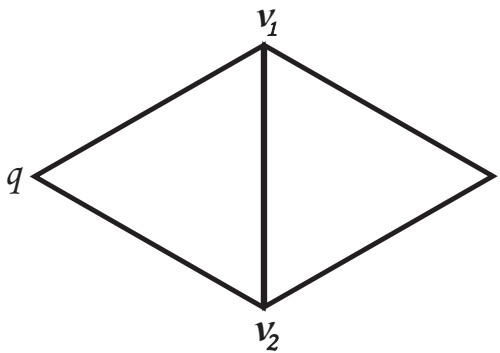}\\
  \caption{A rhombus $C$ formed by joining two congruent equilateral triangles, and a minimizing bisection of $C$}\label{fig:rombo}
\end{figure}

Previous Theorem~\ref{th:main} is not sharp,
in the sense that there exist minimizing bisections
which cannot be identified by means of this result.
This can be easily seen for a circle:
any bisection given by a diameter is optimal,
but the second hypothesis is not verified (observe Remarks~\ref{re:re},\ ii)).
We shall also illustrate this fact in the following Example~\ref{ex:6-gon}.

\begin{example}
\label{ex:6-gon}
Let $C$ be the centrally symmetric hexagon depicted in Figure~\ref{fig:malillo},
obtained by cutting symmetrically two opposite corners of a square
(the lengths of the resulting edges are $3.31$ and $5.66$ units, with non-right angles equal to $3\pi/4$). 
Consider the bisection $P$ determined by the segment joining the midpoints of the shortest edges of $C$.
It can be checked that $d_M(P)=8.17$, provided by the distance between the endpoint $v_1$ and $\phi_{C_1}(v_1)$, see Figure~\ref{fig:malillo2}.
For any other bisection $P'$ given by a straight line passing through the center of symmetry,
it follows that $d_M(P')>d_M(P)$,
since one of the corresponding subsets will contain the segment $\overline{v_1\,\phi_{C_1}(v_1)}$ or $\overline{v_2\,\phi_{C_1}(v_2)}$,
both with length equal to $d_M(P)$, 
or the segment $\overline{x\,\phi_{C_2}(v_2)}$ or $\overline{x\,\phi_{C_2}(v_1)}$, both with length equal to $8.34$
(where $\phi_{C_2}(v_i)$ is the farthest point from $v_i$ in $\ptl C_2$, $i=1,2$).
Therefore, $P$ is a minimizing bisection,
but Theorem~\ref{th:main} cannot be applied because the second hypothesis does not hold, as shown in Figure~\ref{fig:malillo2}.
\end{example}

\begin{figure}[h]
    \includegraphics[width=1\textwidth]{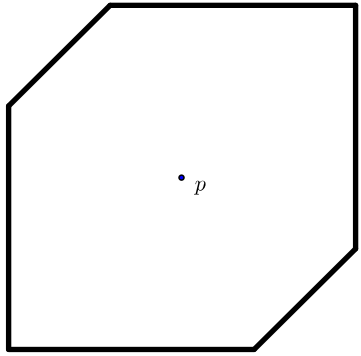}\\[-6mm]
  \caption{Centrally symmetric hexagon obtained by cutting a square}
\label{fig:malillo}
\end{figure}

\begin{figure}[h]
    \includegraphics[width=1\textwidth]{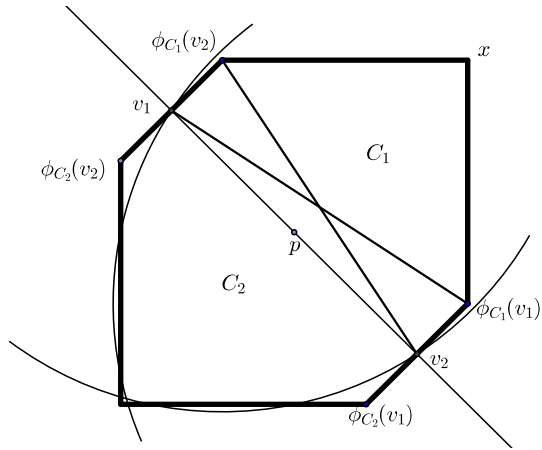}\\
  \caption{Theorem~\ref{th:main} cannot be applied since two pieces of $\ptl C_2$ are not contained in $A_1\cup A_2$}
  \label{fig:malillo2}
\end{figure}

\section{Standard bisection}
\label{se:standard}

In this section we shall introduce a particular bisection for a centrally symmetric planar convex body,
which is called \emph{standard bisection}.
Its construction is analogous to the one described in \cite[\S.~3]{extending},
which concerns the standard $k$-partitions of $k$-rotationally symmetric planar convex bodies (for $k\in\nn$, $k\geq 3$).
We shall emphasize here the different behavior of the standard bisections in our setting
with respect to those ones, in terms of optimality and uniqueness.

\begin{definition}
Let $C\in\C$.
A standard bisection of $C$ is a decomposition of $C$ determined by
two symmetric inradius segments of $C$.
We shall denote it by $P_2(C)$, or simply $P_2$.
\end{definition}

Note that, for a given set in $\C$, it is always possible to construct an associated standard bisection
(due to the existing symmetry),
which will consist of a line segment passing through the center of symmetry.
In fact, it will be one of the shortest chords of the set passing through that point, see Figure~\ref{fig:standards}.

\begin{figure}[h]
  \includegraphics[width=0.7\textwidth]{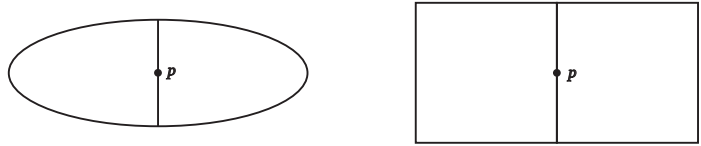}\\
  \caption{Standard bisections for an ellipse and a rectangle}\label{fig:standards}
\end{figure}

As indicated in the Introduction,
it is known \cite[Th.~4.5]{extending} that, for any $k$-rotationally symmetric planar convex body,
its corresponding standard $k$-partition (defined by means of $k$ inradius segments symmetrically placed) is always
minimizing for the maximum relative diameter functional, when $k\geq 3$, see Figure~\ref{fig:some}.
It is then natural to wonder whether the standard bisection is minimizing in
our centrally symmetric case (which corresponds to $k=2$).
This holds for a wide variety of sets of our class,
but it is not true in general, as shown in the following Example~\ref{ex:non}.

\begin{example}
\label{ex:non}
Let $C$ be a rhombus, and consider an associated standard bisection $P_2$ of $C$,
depicted in the left-hand side of Figure~\ref{fig:rombo2}.
It is clear that $P_2$ is not minimizing, since the bisection $P$
determined by the vertical line segment passing through the center of symmetry
(right-hand side of Figure~\ref{fig:rombo2}) has smaller value for $d_M$.
In fact, $P_2$ does not satisfy the necessary condition from Proposition~\ref{prop:necesaria},
and Theorem~\ref{th:main} yields that $P$ is a minimizing bisection of $C$.
\end{example}

\begin{figure}[h]
  \includegraphics[width=0.68\textwidth]{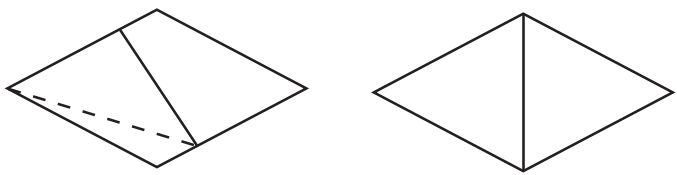}\\
  \caption{The standard bisection $P_2$ of the rhombus is not minimizing}\label{fig:rombo2}
\end{figure}

One might think that the standard bisection from Example~\ref{ex:non}
is not minimizing essentially because the necessary condition
from Proposition~\ref{prop:necesaria} does not hold.
The following example shows that even when this necessary condition is satisfied,
we cannot assure the minimizing character of a given standard bisection.

\begin{example}
\label{ex:salvador}
Let $S$ be a square, and call $v_1$, $v_2$ the midpoints of the upper and lower edges,
and $w_1$, $w_2$ the midpoints of the other two edges,
see Figure~\ref{fig:contraejemplos}.
Consider $C=S\cap B(v_1,d(v_1,v_2))\cap B(v_2,d(v_1,v_2))$, which is a centrally symmetric planar convex body.
It is clear that the bisection $P_2$ of $C$ provided by the segment $\overline{v_1\,v_2}$ is standard,
as well as the bisection $P_2'$ given by $\overline{w_1\,w_2}$.
Both of them satisfy the necessary condition from Proposition~\ref{prop:necesaria},
but we have that $d_M(P_2')=d(w_1,x)>d(v_1,v_2)=d_M(P_2)$, and so $P_2'$ is not minimizing.
Moreover, Theorem~\ref{th:main} implies that $P_2$ is a minimizing bisection for $d_M$.
\end{example}

\begin{figure}[htp]
\centering{
\subfigure{\includegraphics[width=0.47\textwidth]{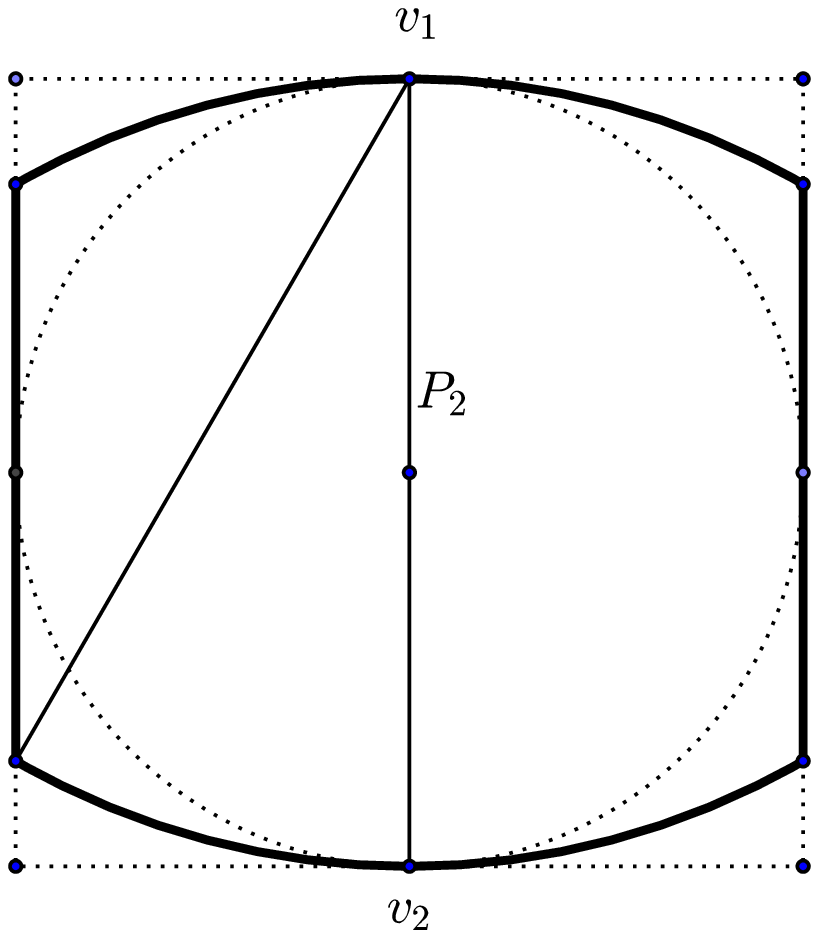}}
\hspace{0.01\textwidth}
\subfigure{\includegraphics[width=0.47\textwidth]{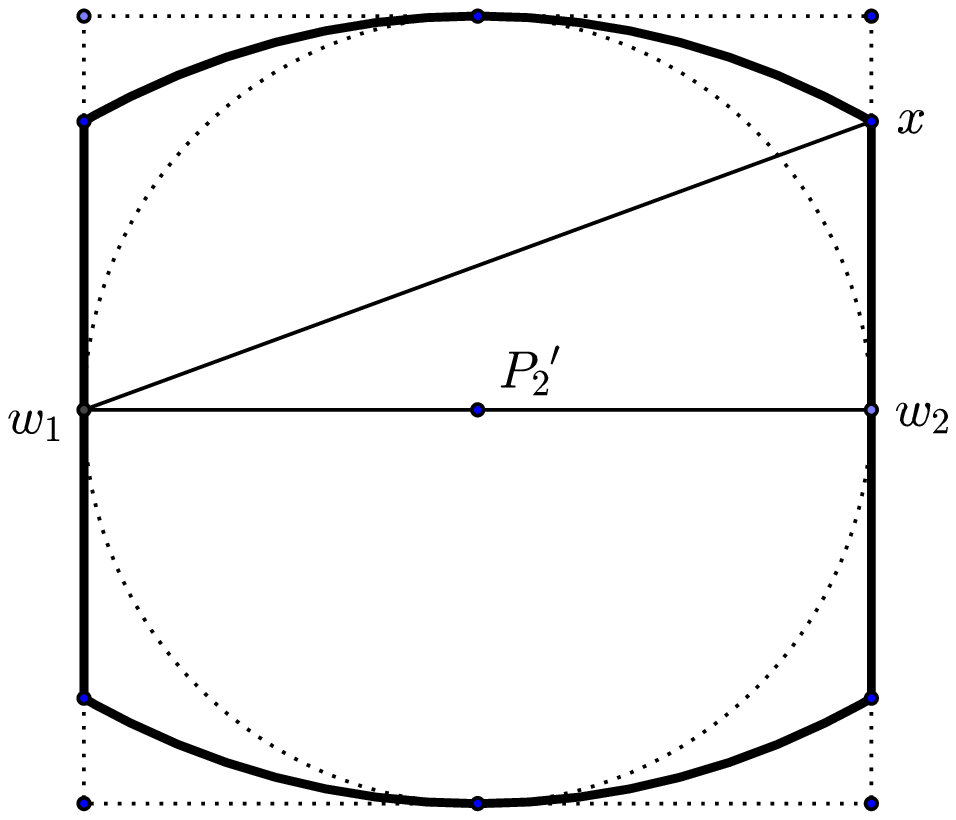}}}
\caption{$P_2$ and $P_2'$ are two standard bisections of $C$. We have that $P_2$ is minimizing, while $P_2'$ is not}
\label{fig:contraejemplos}
\end{figure}

The two previous Examples~\ref{ex:non} and~\ref{ex:salvador} reveal that
a standard bisection is not optimal in general.
Although some partial results can be obtained in some restrictive situations,
we shall refer to Theorem~\ref{th:main} in order to determine if a given one is minimizing.

\begin{remark}
Consider $C\in\C$ and a standard bisection $P_2$ of $C$, with endpoints $v_1$, $v_2\in\ptl C$.
If $d_M(P_2)=d(v_1,v_2)$, then $P_2$ is necessarily minimizing
(recall that, in this case, we cannot apply Theorem~\ref{th:main}, see Remarks~\ref{re:re},\ ii)).
The reason is that for any other bisection $P$ determined
by a straight line passing through the center of symmetry $p\in C$, with endpoints $w_1$, $w_2\in\ptl C$,
it follows that $d(w_1,w_2)\geq d(v_1,v_2)$, since $\overline{p\,v_i}$ is an inradius segment of $C$, $i=1,2$,
and so $d_M(P)\geq d(w_1,w_2)\geq d(v_1,v_2)=d_M(P_2)$.
This property does not hold for bisections which are not standard:
if we consider an ellipse $C$, and the bisection $P$ determined by the segment $\overline{v_1\,v_2}$,
where $D(C)=d(v_1,v_2)$, then we clearly have that $P$ is not minimizing, although $d_M(P)=d(v_1,v_2)$.
\end{remark}

\subsection{Uniqueness of the standard bisection}
\label{sub:uniqueness}
In general, the standard bisection of a centrally symmetric planar convex body is not uniquely defined:
we clearly have two different ones for a given square (joining the midpoints of each pair of opposite edges),
and an infinite amount of them for a circle (provided by the diameter segments).
In these two cases, the maximum relative diameter of the different standard bisections coincide,
and so this fact is not relevant for our optimization problem.
However, the lack of uniqueness may also refer to the values of the maximum relative diameter,
as shown in the following Example~\ref{ex:capbody}.

\begin{example}
\label{ex:capbody}
Let $C$ be a planar \emph{cap body}, that is,
the convex hull of a circle and two exterior symmetric points with respect to the center
(which will be called the vertices of $C$).
This centrally symmetric planar convex body possesses
an infinite quantity of associated standard bisections,
determined by each pair of symmetric points lying in the circular pieces of $\ptl C$.
In this setting, if the vertices of $C$ are far enough from the center of the circle,
all the standard bisections of $C$ will have \emph{different values} for the maximum relative diameter.
For instance, for the two standard bisections from Figure~\ref{fig:capbody},
the maximum relative diameter equals the distance
between an endpoint of the bisection and a vertex of the cap body,
thus attaining distinct values.
We point out that the same happens for the standard bisections of the set from Example~\ref{ex:6-gon},
as indicated in Section~\ref{se:examples} below.
\end{example}

\begin{figure}[ht]
    \includegraphics[width=0.7\textwidth]{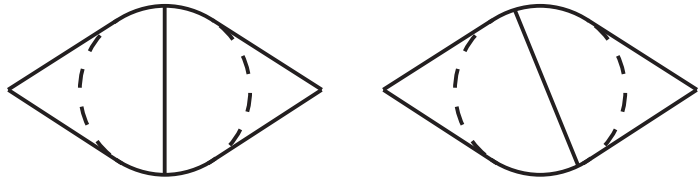}\\
  \caption{Two standard bisections with different values for $d_M$}\label{fig:capbody}
\end{figure}

\begin{remark}
The behavior described in Example~\ref{ex:capbody} is another remarkable peculiarity
of our problem with respect to the analogous one for $k$-rotationally symmetric planar convex bodies ($k\geq3$),
where different standard $k$-partitions always yield the same value
for the maximum relative diameter, due to~\cite[Lemma~3.2]{extending}.
\end{remark}

\begin{remark} For a given set $C$ in $\C$,
the standard bisection of $C$ is uniquely defined
if and only if the associated inball touches $\ptl C$ only twice.
\end{remark}

\section{Some Examples}
\label{se:examples}

In this section we collect several examples of centrally symmetric planar convex bodies,
indicating one of the minimizing bisections in each case.

The corresponding standard bisections are minimizing for the square, the rectangle or the ellipse,
by direct application of Theorem~\ref{th:main}.

The case of the circle is special, since the maximum relative diameter functional
is constant for \emph{any} arbitrary bisection (such a constant is the diameter of the circle).
Therefore, any bisection of the circle can be considered minimizing.

\begin{figure}[h]
  \includegraphics[width=0.88\textwidth]{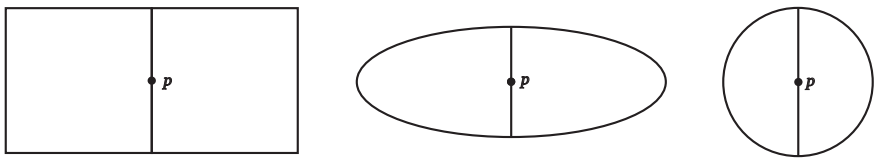}\\
  \caption{Minimizing bisections for the rectangle, the ellipse and the circle}
\end{figure}

For the hexagon treated in Example~\ref{ex:6-gon}, which is depicted in Figure~\ref{fig:malillo},
we have already described a minimizing bisection.
We point out that such a bisection is standard, and that
there are two other standard bisections for this set
(joining each pair of larger opposite symmetric edges), which are not minimizing
(it can be checked that the necessary condition from Proposition~\ref{prop:necesaria} does not hold).
Recall that Theorem~\ref{th:main} cannot be applied for this optimal bisection.

We have studied the rhombus in Example~\ref{ex:non}. The standard bisection is not minimizing,
and Theorem~\ref{th:main} yields that the bisection given by a vertical straight line passing through the
center of symmetry minimizes $d_M$, see Figure~\ref{fig:rombo2}.

For the cap body from Example~\ref{ex:capbody},
we have already indicated that there are an infinite amount of associated standard bisections,
each of them with a different value for the maximum relative diameter when the vertices are
far enough from the center of symmetry, see Figure~\ref{fig:capbody}.
In these situations, among all of them, the one determined by a vertical straight line passing through the center of symmetry
is the unique minimizing bisection, by applying Proposition~\ref{prop:necesaria} and Theorem~\ref{th:main}.
We point out that certain variations of this set provide examples with infinite standard bisections,
being none of them minimizing.

\begin{figure}[h]
  \includegraphics[width=0.73\textwidth]{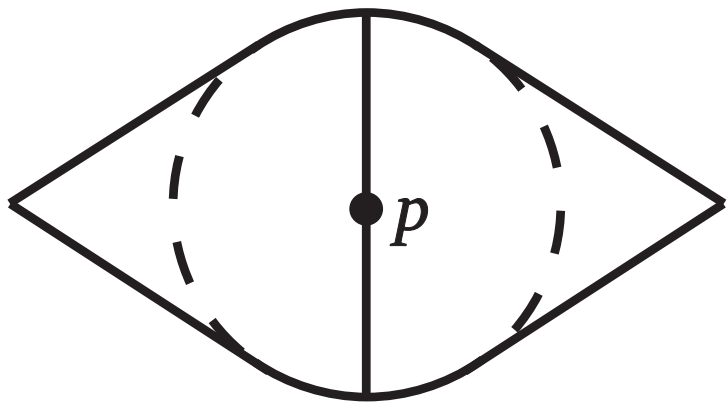}\\
  \caption{Minimizing bisection for a cap body}
\end{figure}

Finally, for the centrally symmetric planar convex body from Figure~\ref{fig:cut-ellipse},
obtained by cutting symmetrically a given ellipse, we have that the associated standard bisection is not minimizing.
By using Theorem~\ref{th:main}, it follows that a minimizing bisection is the one shown in Figure~\ref{fig:cut-ellipse}.

\begin{figure}[h]
  \includegraphics[width=0.73\textwidth]{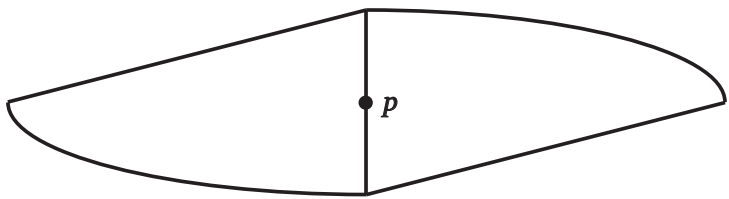}\\
  \caption{Minimizing bisection for this centrally symmetric planar convex body}\label{fig:cut-ellipse}
\end{figure}

\section{Some remarks}
\label{se:comments}

We finish these notes with some comments related with our optimization problem.

\subsection{Optimal set. } 
Another interesting question for this problem is searching for the \emph{optimal sets},
that is, the centrally symmetric planar convex bodies \emph{of unit area} with
the \emph{minimum} possible value for the maximum relative diameter functional.
The unit-area condition here is required just as a normalization for the sets of our class. 
In this setting, the optimal set is unique and has been obtained in~\cite[Example~2.3 and Th.~5]{mps}: 
it consists of the intersection of a certain strip (delimited by two parallel lines) and a symmetric lens.

\subsection{Dual problems. } There are some dual optimization problems to the one discussed in this paper,
but they are worthless since their solutions are trivial.
For instance, if we are interested in the bisections attaining the \emph{maximum} possible value for $d_M$,
it is clear that we can consider a bisection determined by a diameter segment of the set
(and so, the diameter will be such maximum value).
In fact, this will happen for any bisection with a subset
containing two points whose distance equals the diameter of the set.

On the other hand, we can consider the \emph{minimum relative diameter} functional, defined as
$$d_m(P)=\min\{D(C_1),D(C_2)\}, $$
where $P$ is a bisection with subsets $C_1$, $C_2$.
This functional has been already studied in some previous works, see~\cite{cms,css}.
It is easy to check that $d_m$ tends to zero for bisections with one of its associated subsets being reduced to a point,
and that its maximum value will be attained again by a bisection given by a diameter segment.

\subsection{Relation with the Borsuk number. } For a given $C\in\C$,
the optimization problem for the maximum relative diameter functional treated in this paper is meaningless when
that functional is constant over all the bisections of $C$ 
(in that case, all the bisections can be seen as minimizing).
This situation only happens when $C$ is a circle, and it is equivalent to the following property:
the unique centrally symmetric planar convex body with
Borsuk number equal to three is the circle (see~\cite{danuta, number} and references therein for details on this question).

\noindent {\bf Acknowledgements.} 
The first author is partially supported by the project MTM2013-48371-C2-1-P (Ministerio de Econom\'ia e Innovaci\'on),
and by Junta de Andaluc\'ia grant FQM-325 (Consejer\'ia de Econom\'ia, Innovaci\'on, Ciencia y Empleo).
The second author is partially supported by MINECO/FEDER project MTM2015-65430-P and ``Programa de Ayudas a Grupos de
Excelencia de la Regi\'on de Murcia'', Fundaci\'on S\'eneca, 19901/GERM/15.

\end{document}